\documentclass[12pt,a4paper,report]{article}
\usepackage{amsmath,amssymb,amsthm,amsfonts,amscd,euscript,verbatim, t1enc, newlfont}
\usepackage[affil-it]{authblk}
\usepackage{tikz-cd}
\usepackage{hyperref}
\usepackage{graphicx}
\usepackage[all]{xy}
\usepackage{color}
 \theoremstyle{plain}
\newtheorem{theo}{Theorem}[section]

\newtheorem{pr}[theo]{Proposition}
\newtheorem{question}[theo]{Question}

 \newtheorem{lem}[theo]{Lemma}
 \newtheorem{coro}[theo]{Corollary}
  
\theoremstyle{remark}
\newtheorem{rema}[theo]{Remark}

\theoremstyle{definition}
\newtheorem{defi}[theo]{Definition}

\newtheorem*{notat}{Notation}
\newtheorem*{notata}{Notation and conventions}

\newcommand\dmgm{DM_{gm}}

\newcommand\dmges{DM_{gm}^{eff}}

\newcommand\obj{\operatorname{Obj}}

\newcommand\id{\operatorname{id}}

\newcommand\cu{\underline{C}}
\newcommand\du{\underline{D}}

\newcommand\ssets{\underline{sSets}_{\bullet}}

\newcommand\au{\underline{A}}
\newcommand\bu{\underline{B}}

\newcommand\hu{\underline{H}}

\newcommand\hrt{{\underline{Ht}}}
\newcommand\hw{{\underline{Hw}}}
\newcommand\hwoo{{\underline{Hw}_{\infty}}}

\newcommand\z{{\mathbb{Z}}}
\newcommand\q{{\mathbb{Q}}}

\newcommand\chow{Chow}

\newcommand\spt{\underline{Spt}_{\bullet}}
\newcommand\map{\mathcal{M}}

 \DeclareMathOperator\kthe{\operatorname{K}}

\DeclareMathOperator\enom{\operatorname{End}}

\DeclareMathOperator\kar{\operatorname{Kar}}

\numberwithin{equation}{subsection}

\begin{document}

  \title{Theorem of the heart in negative K-theory for weight structures} 
 \author{Vladimir Sosnilo
 \thanks{ 
the author was  supported by the Russian Science Foundation grant N 16-11-10200.
\newline University of Duisburg-Essen, Faculty of Mathematics, Campus Essen, Thea-Leymann-Strasse 9, Essen, 45127 Germany
\newline Chebyshev Laboratory, St. Petersburg State University, 14th line, 29B, Saint Petersburg,  199178 Russia
\newline
} }
\maketitle

\abstract 
We construct the strong weight complex functor (in the sense of Bondarko) for a stable infinity category $\cu$ equipped with 
a bounded weight structure $w$. 
Along the way we prove that $\cu$ is determined by the infinity-categorical heart of $w$. 
This allows us to compare the K-theory of $\cu$ and the K-theory of $\hw$, the classical heart  of $w$. In particular, 
we prove that $\kthe_{n}(\cu) \to \kthe_{n}(\hw)$ are isomorphisms for $n \le 0$.

\bigskip
 
\tableofcontents


\section*{Introduction}
\label{sec:intro}
\addcontentsline{toc}{section}{\nameref{sec:intro}}

The concept of weight structures on triangulated categories was introduced by 
Bondarko in \cite{bws} and also independently by 
Pauksztello in \cite{pauk} by the name of co-t-structures. 
Bondarko's reason for introducing weight structures was to study
various triangulated categories of motives. 
There exist Chow weight structures on the categories 
$\mathcal{DM}(S;R)$, $\mathcal{DK}(S;R)$ for nice pairs $(S,R)$ where 
of $S$ is a scheme and $R$ is a ring  
(see \cite{relweights} and \cite{boluzgarev}). 
This weight structure also exists on $\dmges(k;R)$ (resp. $\dmgm(k;R)$) for perfect 
fields $k$ if the characteristic of k is either 0 or is invertible in $R$
(see \cite{bws}). The heart  
of this weight structure is the category of effective (resp. non-effective) 
Chow motives. 
There also exist Gersten weight structures on certain categories of pro-motives 
containing as a full subcategory either   
$SH^{S^1}(k)^c$, $SH^{MGL}(k)^c$, or $SH(k)^c$ (see \cite{gerst}). 

A weight structure on a triangulated category consists of two subclasses of 
"non-positive" and "non-negative" objects of the category   
that satisfy certain axioms similar to the axioms of t-structures. 
Both t-structures and weight structures have the associated hearts and both 
are used to reduce studying arbitrary 
objects of a triangulated category to studying objects of the heart. 
However, while t-structures let one work with a triangulated 
category 
as with the derived category of some abelian category, weight structures 
are designed to let one 
work with a triangulated category as with the homotopy category of 
complexes over some additive category. 
With a weight structure $w$ on a triangulated category $\cu$ one gets 
a lot of methods to study the category. 
As the easiest application of the theory one gets 
spectral sequences $E(F,M)$ for any homological functor $F$ on $\cu$ and any object 
$M$ of $\cu$. The spectral sequences are functorial in $M$ 
starting from the second page 
and for weight-bounded objects $M$ they converge to $F(M)$. 
For example, in the case of the Gersten weight structures these spectral sequences 
are exactly the Coniveau spectral sequences. 
If $\cu$ satisfies the Brown representability theorem and the classes 
defining the weight structure are closed under taking all coproducts then $\cu$ 
admits a certain t-structure called adjacent to $w$. 

One of the greatest features of the theory is the existence of the so-called weak 
weight complex functor. Ideally we would like to have a conservative 
triangulated functor from $\cu$ to $K(\hw)$, the homotopy category of complexes over 
the heart of $w$, mapping non-positive (resp. non-negative) objects into complexes 
homotopy equivalent to complexes concentrated in non-positive (resp. non-negative) 
degrees and inducing an equivalence of the hearts. In \cite{wcom} Bondarko
constructed this 
strong weight complex functor for categories that admit a filtered enhancement or 
a negative dg-enhancement and conjectured that it exists in general. 
Besides, in \cite{bws} he was also able to construct a weak version of this functor 
for any $\cu$. It's defined as follows. 
The weak category of complexes $K_w(\hw)$ is the 
quotient of $K(\hw)$ by a certain ideal of morphisms weak homotopic to 0. 
Unfortunately this category is not even triangulated, but there are still notions 
of non-positive and non-negative objects and of distinguished triangles. 
The weak weight complex functor is a conservative functor $\cu \to K_w(\hw)$ 
satisfying analogous properties to that of the strong weight complex functor. 
Using this functor one can extend any additive functor from 
$\hw$ to an abelian category to a homological functor on $\cu$ (see 
Corollary 2.3.4 of \cite{gerst} and Theorem 2.3 of \cite{kelsai}). 
Its existence is also used to show that $\kthe_0(\cu)$ is 
isomorphic to $\kthe_0(\hw)$ if $w$ is a bounded weight structure (see Theorem 
8.1.1 of \cite{bws}). 
Although this weak weight complex functor is already a very useful technique, we 
still want to construct the strong weight complex functor. In this paper we do that 
in the case $\cu$ has an $\infty$-categorical enhancement and $w$ is either 
bounded or compactly-generated.

Let $\cu$ be a stable infinity-category endowed with a bounded t-structure $t$ on its 
homotopy category. 
In this setting Clark Barwick proved an analogue of Neeman's theorem of the heart 
(see \cite{thofhrt}). 
More precisely, he showed that the natural map $\kthe^{con}(\hrt) \to 
\kthe^{con}(\cu)$ is a homotopy 
equivalence of connective $\kthe$-spectra, where $\hrt$ is the heart of the 
t-structure. 
Moreover, in \cite{thofneghrt} the map $\kthe(\hrt) \to \kthe(\cu)$ of nonconnective 
K-theory spectra was also shown to be 
an equivalence in case the heart $\hrt$  is noetherian. They also conjecture that 
the map should be an equivalence in general.  
Besides, they prove that the map $\kthe_{-1}(\hrt) \to \kthe_{-1}(\cu)$ 
is an isomorphism without any extra assumptions.

Now assume $\cu$ is a stable infinity-category endowed with a bounded weight 
structure $w$ on its homotopy category. 
As it has already been mentioned there always exists an 
isomorphism $\kthe_0(\hw) \cong \kthe_0(\cu)$. 
It is natural to ask whether an analogue of the theorem of the heart holds for weight 
structures. 
The answer to this question turns out to be no in general as there are 
counterexamples. 
However, in this paper we construct natural maps $\kthe_{n}(\cu) \to \kthe_{n}(\hw)$ for all $n$ 
and prove that they are isomorphisms for $n \le 0$. 
The way we construct the maps is the following. 
We first construct the 
strong weight complex functor on the level on $\infty$-categories $\cu \to Com^b(\hw)$. 
And then this functor induces the natural maps in K-theory. 

Combining this result with the recent result of \cite{thofneghrt} to the setting of 
triangulated categories of motives we obtain that the non-triviality of the negative 
K-groups of the additive category $\chow(S)$ would yield an obstruction to having a 
motivic t-structure on $\dmgm$.  

The paper is organized as follows. In sections \ref{1} and \ref{2} we remind the 
reader of basic notions that we use in the paper, their basic properties, and state the main 
theorems that we will use. We also introduce some notations there. 
In section \ref{3} we consider the functor from the $\infty$-category $WCat^{st,b}_{\infty}$  of stable $\infty$-categories equipped with a weight structure 
to the $\infty$-category $Cat^{add}_{\infty}$
of additive $\infty$-categories given by taking the $\infty$-heart $\hwoo$ of a weight structure. We prove that this functor is an equivalence 
onto its image. It consists of those additive $\infty$-categories in which idempotents of a certain type split. This allows us to construct the 
weight complex functor for any $\cu$ with a bounded weight structure $w$ as the functor corresponding 
to the map $\hwoo \to Nerve(\hw)$ via this equivalence. 
In section \ref{4} we prove that the weight complex induces isomorphisms in negative K-groups. 
In the last section we discuss relations between the negative K-theory of the category of Chow motives, 
the existence of the motivic t-structure, and also the smash-nilpotence conjecture.

The author is deeply grateful to Benjamin Antieau, Mikhail Bondarko, Adeel Khan, and Marc Levine  for numerous useful discussions, 
as well 
as to Xindi Ai, Tom Bachmann, and Maria Yakerson for proofreading the text and pointing out many mathematical and linguistic mistakes. 
He also wants to thank Elden Elmanto, whose questions motivated the author to think about the subject of the paper. 

\section{Reminder on weight structures and infinity-categories}\label{1}

\begin{notata}
First we fix some notations on basic category theory notions. 
Categories can be large or small (e.g. in the sense of a Grothendieck universe of large sets containing the Grothendieck universe of small sets). 

\begin{itemize}
\item
$Sets$ is the category of sets.

\item
$Cat$ is the category of small categories.

\item
$Top$ is the category of topological spaces.

\item
If $C$ is a category, then $\obj C$ denotes the class of objects of $C$. 
For any objects $X,Y$ of $C$ the set of morphisms is denoted by $C(X,Y)$ or by $Hom(X,Y)$ if the category is clear from the context. 


\item
We use the homological grading for complexes. 

\end{itemize}
\end{notata}

Recall some notions from the theory of simplicial sets. 
Denote by $\Delta$ the category whose objects are the linearly ordered sets $[n]= \{0,\cdots, n\}$ for each $n \in \mathbb{N} \cup \{0\}$ and 
the morphisms 
between two such sets are order-preserving maps of sets.  
A simplicial set is a presheaf of sets on the category $\Delta$. The category of simplicial sets is denoted by $sSets$. 
Recall that equivalently a simplicial set can be defined as a collection of sets $\{X[i]\}_{i\ge 0}$ with maps 
$s_i:X([n-1]) \to X([n])$ for all $0\le i<n$ and $d_i:X([n]) \to X([n-1])$ for all $0 \le i \le n$ satisfying certain relations. 
One can also define simplicial sets with values in any category. The category of simplicial abelian groups is denoted by $sAb$. 

For any $n\ge 0$ there is a simplicial set $\Delta^n$  given by the presheaf represented by $[n]$. 
For $n=0$ we also denote it by $pt$.
The $n+1$ order-preserving embeddings $[n-1] \to [n]$ give rise to $n+1$ maps of simplicial sets $d_i: \Delta^{n-1} \to \Delta^n$. 
We denote the image of $d_i$ by $\partial_i\Delta^n$. 
For any $0 \le k \le n$ the simplicial set $\Lambda^n_k$ is the union of the subsets of $\partial_i\Delta^n$ inside $\Delta^n$ for all $i \neq k$. 
We denote by $\partial\Delta^n$ the union of  the subsets of $\partial_i\Delta^n$ inside $\Delta^n$ for all $i$. 


There is a functor $|-|:sSets \to Top$ called the geometric realization of a simplicial set. It's defined as the left Kan 
extension of the 
functor $\Delta \to Top$ sending $[n]$ to the topological simplex $\Delta^n_{top}$. 
We call a map of simplicial sets $X \to Y$ a homotopy equivalence\footnote{To 
not confuse it with other notions we will always use the term "homotopy equivalence", i.e. not "equivalence" or "weak equivalence".}
 if the associated map $|X| \to |Y|$ is a weak homotopy 
equivalence.

The category admits a closed monoidal structure given by the cartesian product and the mapping simplicial set functor 
$Map(X,Y)([n]) = Hom(X \times \Delta^n, Y)$. 
A simplicial category is a category enriched over simplicial sets. 


\subsection{Quasi-categories}

\begin{defi}
A simplicial set $X:\Delta^{op} \to Sets$ is called a quasi-category if the following lifting property is satisfied 

\begin{center}
\begin{tikzcd}
\Lambda^n_k \arrow[d, hook] \arrow[r, "\forall"] & X \\
\Delta^n \arrow[ur, dashed,"\exists"]
\end{tikzcd}
\end{center}
for any $0<k<n$, i.e. for any map $s:\Lambda^n_k \to X$ there exists a map $\Delta^n \to X$ extending $s$. 

Recall that $X$ is called a Kan complex if the lifting property is satisfied for any $0\le k\le n$.
\end{defi}

Recall that the notion generalizes the notion of usual categories via the following construction.  

Any poset gives rise to a category, hence the category $\Delta$ can be viewed as a full subcategory of the category $Cat$. 
Then any small category $C$ gives rise to a presheaf $Nerve(C)$ on $\Delta$ given by 
$[n] \mapsto Cat([n],C)$. This presheaf can be shown to be a quasi-category as the lifting property in this case 
boils down to the existence of compositions of composable morphisms. 
We will call $Nerve(C)$ the nerve of $C$. 
This yields a fully faithful functor $Cat \to Cat_{\infty}$ where $Cat_{\infty}$ is the full simplicial subcategory of the category of simplicial sets whose objects are quasi-categories. 

Moreover, there are other interesting examples of quasi-categories. 
Denote by $P_{i,j}$ the category corresponding to the partially ordered set $\{I \subset [i,j] \subset \mathbb{N}\cup \{0\} | i,j\in I\}$, where the partial order is given by inclusion. 
Denote by $C(n)$ the simplicial category whose set of objects is $\{0,\cdots, n\}$. 
$$Map_{C(n)}(i,j) = 
\begin{cases}
N(P_{i,j}), & i\le j\\
\emptyset, & i>j
\end{cases}$$
This gives a functor $\Delta \to sSet-Cat$.
Now for any small simplicial category $S$ its coherent nerve $\underline{Nerve}(S)$ is the restriction of the functor represented by $S$ to $\Delta$. 
In other words it is the functor $[n] \to Map_{sSet-Cat}(C(n), S)$. In general, it's not a quasicategory, but it is if $S$ is enriched 
over the category of Kan complexes (see Proposition 1.1.5.10 of \cite{htt}). 
In this sense the category of pointed Kan complexes gives rise to a quasi-category that we will denote by $\ssets$. 

Throughout the paper we will use quasi-categories as models for $(\infty,1)$-categories, although any 
other reasonable model would serve our properties.  
From now on by an $\infty$-category we will mean a quasi-category. 
$1$-simplices of a quasi-category will be called morphisms and its $0$-simplices will be called objects. Many notions of usual category theory along with their basic properties can be generalized to quasi-categories. We list them here referring to \cite{htt} and \cite{qcjoy} for 
details. 

\begin{itemize}
\item
A full subcategory of a quasi-category $\au$ is a subsimplicial set $\au'$ 
such that any simplex $x$ of $\au$ all of whose boundaries belong to $\au'$ also belongs to $\au'$. 
Clearly, this condition is equivalent to the following lifting property. 
\begin{center}
\begin{tikzcd}
\partial\Delta^n \arrow[d, hook] \arrow[r, "\forall"] & \au'\arrow[d,hook] \\
\Delta^n \arrow[ur, dashed,"\exists"]\arrow[r,"\forall"] & \au
\end{tikzcd}
\end{center}
It's also a quasi-category. 

For any set of objects $X\subset \au([0])$ one can define the full subcategory spanned by $X$ as the minimal subsimplicial set of $\au$ 
containing $X$ and satisfying the lifting property.

\item 
For a quasi-category $\au$ any pair of objects $(x,y) \in \au([0]) \times \au([0])$ the simplicial set of morphisms $\au(x,y)$ is the pullback of the following diagram 
\begin{center}
\begin{tikzcd}
 & pt\arrow[d, "{(x,y)}"] \\
Map(\Delta^1,\au) \arrow[r] & \au([0]) \times \au([0])
\end{tikzcd}
\end{center}
The simplicial set is always a Kan complex.

\item
There is a functor $h:Cat_{\infty} \to Cat$ left adjoint to the ordinary nerve functor. The category $h(\au)$ is called the homotopy category of $\au$. 
Its set of objects is $\obj \au$ and the set of morphisms between $x,y \in \obj h(\au)$ is $\pi_0\au(x,y)$ (see \cite{qcjoy}(1.6-1.8)).

\item 
For any simplicial set $K$ and a quasi-category $\bu$ the mapping space $Map(K,\bu)$ is a quasi-category. 
In this case we also call it the category of functors $Fun(K,\bu)$. 
If $\bu$ is a full subcategory of a quasi-category $\bu'$ then the subsimplicial subset $Fun(K,\bu) \subset Fun(K,\bu')$ is also a full 
subcategory. 

Let $\au,\bu$ be quasi-categories. 
A functor $F \in \obj Fun(\au,\bu)$ is called fully faithful if the natural map $\au(x,y) \to \bu(F(x),F(y))$ is 
a homotopy equivalence for any pair $(x,y) \in \obj \au \times \obj \au$.

A functor $F \in \obj Fun(\au,\bu)$ is called essentially surjective if the induced map $h(\au) \to h(\bu)$ is essentially surjective. 

A functor $F \in \obj Fun(\au,\bu)$  is called a (categorical) equivalence if it is fully faithful and essentially surjective.  

All these notions are compatible with the classical notions for $\pi_0(F)\in \obj Fun(h(\au),h(\bu))$, i.e. if $F$ 
has one of the properties above then $h(F)$ has the corresponding property too.

\item 
There is a notion of adjoint functors, so is a notion of a limit of a diagram. Adjoint functors yield adjoint functors on the homotopy categories. 
We say that a category $\au$ admits finite limits (resp. colimits) if any diagram $K \to \au$  has a limit (resp. a colimit), where $K$ is a 
simplicial set having only 
finitely many non-degenerate simplices. 

Let $\au,\bu$ be quasi-categories that have finite colimits (resp. limits, resp. both). 
Functors between $\au$ and $\bu$ preserving finite colimits 
(resp. limits, resp. 
both) form a subcategory of the category of functors which we denote by $Fun_{rex}(\au,\bu)$ (resp. $Fun_{lex}(\au,\bu)$, resp. 
$Fun_{ex}(\au,\bu)$). 

Let $\au,\bu$ be quasi-categories that have small colimits (resp. limits). Functors between $\au$ and $\bu$ that preserve 
colimits (resp. limits) form a subcategory of the category of functors which we denote by $Fun^L(\au,\bu)$ 
(resp. $Fun^R(\au,\bu)$). 

\item 
For any quasi-category $\au$ the dual quasi-category $\au^{op}$ is given by the simplicial subset whose underlying sets are the same but 
the face and degeneracy maps are reordered, that is 
$$d^{\au^{op}}_i = d_{n-i}^{\au}:\au([n]) \to \au([n-1])$$ 
and 
$$s^{\au^{op}}_i = s_{n-i}^{\au}:\au([n]) \to \au([n+1])$$

The homotopy category of the dual category $h(\au^{op})$ is equal to the dual category of the homotopy category $h(\au)^{op}$. 

\item 
For any small pointed quasi-category $\au$ there is a fully faithful  left exact functor $\au \to Fun(\au^{op},\ssets)$ denoted by $j$ (see Propositions 5.1.3.1 and 5.1.3.2 of \cite{htt}). We call this functor the Yoneda embedding.

\end{itemize}

\subsection{Stable quasi-categories}

\begin{defi}
An $\infty$-category $\au$ is called stable if it admits a zero object, contains all finite limits and colimits, and any commutative square is a pullback if and only if it is a pushout. 
\end{defi}

Stable $\infty$-categories are important because of the following theorem

\begin{theo}[\cite{ha}, 1.1.2.15]
The homotopy category of a stable $\infty$-category $\cu$ has a natural structure of a triangulated category. 
The shift functor is induced by an adjoint pair of functors $\Sigma:\cu \substack{\leftarrow\\[-1em] \rightarrow} \cu: \Omega$. 
For any pullback square 
\begin{center}
\begin{tikzcd}
X \arrow[r,"f"]\arrow[d] & Y\arrow[d,"g"] \\
pt \arrow[r] & Z
\end{tikzcd}
\end{center}
there is a triangle $X \stackrel{\overline{f}}\to Y \stackrel{\overline{g}}\to Z \to \Sigma X$

\end{theo}

Due to this theorem stable $\infty$-categories are used extensively as enhancements for triangulated categories. 
Most of the known triangulated categories admit such an enhancement (although there are counterexamples and such an enhancement 
does not have to be unique; see \cite{twe} and \cite{ktriang}, respectively).

We also point out the following easy properties 

\begin{pr}[\cite{ha}, 1.1.3.1]
If $K$ is any 
simplicial set and $\cu$ is a stable $\infty$-category then $Fun(K,\cu)$ is stable.
\end{pr}

\begin{pr}[\cite{ha}, 1.1.4.1]
Let $\cu$ and $\cu'$ be stable $\infty$-categories. Then $Fun_{lex}(\cu,\cu') = Fun_{ex}(\cu,\cu') = Fun_{rex}(\cu,\cu')$. 
\end{pr}

Now we give an important example of a stable $\infty$-category.

Abusing the notation we denote by $\z \times \z$ the category associated to the partial order $\z \times \z$. 
Let $\underline{PSpt}_{\bullet} \subset Fun(N(\z \times \z), \ssets)$ be the full $\infty$-subcategory whose objects are functors that 
map objects $(i,j)$ 
into the zero object $pt$ in $\ssets$ for $i\neq j$. 
For any $i \in \z$ and any $F \in \obj \underline{PSpt}_{\bullet}$ we have a diagram 
\begin{center}
\begin{tikzcd}
pt \arrow[r] & F(i+1,i+1) \\
F(i,i)\arrow[u] \arrow[r] & pt \arrow[u]
\end{tikzcd}
\end{center}

Such $F$ is called a spectrum if the squares are pullback for all $i\in \z$. 
The full $\infty$-subcategory of $\underline{PSpt}_{\bullet}$ whose objects are spectra is called the $\infty$-category of spectra $\spt$. 
There exists a natural functor $\ssets \to \spt$ 
adjoint to the functor $\Omega^{\infty}$ of evaluating the functor defining a spectrum at $(0,0)$ (see Corollary 5.5.2.9 of \cite{htt}). 
The full subcategory of $\spt$ containing the image of this functor and which is closed under small colimits is denoted by $\spt^{cn}$. 

\begin{pr}[\cite{ha}, 1.4.3.6]
The category $\spt$ is stable. Its homotopy category is the stable homotopy category $SH$.
\end{pr}


We will give another example of a stable category that we will use extensively later.

For any additive category $A$ the category $Com(A)$ is enriched over the category of simplicial sets. 
 By Corollary 1.3.2.12 of \cite{ha} 
the mapping spaces there are automatically Kan fibrations. 
Now we can consider the simplicial nerve of this category.  

\begin{pr}[\cite{ha}, 1.3.2.10]
The $\infty$-category $\underline{Nerve}(Com(A))$ is stable. Its homotopy category is the homotopy category of complexes. 
\end{pr}

There is also a well-developed theory of localizations of stable $\infty$-categories.

\begin{defi}
Let $\cu$ be a stable $\infty$-category. 
Let $\du \subset \cu$ be a full stable subcategory. Then the bottom right vertex in a pushout diagram
\begin{center}
\begin{tikzcd}
\cu \arrow[r]\arrow[d] & \du\arrow[d] \\
pt  \arrow[r] & \cu/\du 
\end{tikzcd}
\end{center}
is called the localization of $\cu$ by $\du$. 
\end{defi}

By Proposition 5.14 of \cite{highalgkth} the homotopy category $h(\cu/\du)$ is canonically equivalent to the Verdier quotient $h(\cu)/h(\du)$.

\subsection{Additive quasi-categories}

\begin{defi}
An $\infty$-category $\au$ is called additive if its homotopy category is additive. 
\end{defi}

In particular, the nerve of a classical additive category is an additive $\infty$-category. 
Moreover, any stable $\infty$-category is additive. 

Let $\au$, $\bu$ be $\infty$-categories that admit finite products. Then the category $Fun_{add}(\au, \bu)$ is the full 
subcategory of $Fun(\au,\bu)$ 
whose objects are product-preserving functors. 
The subcategory of $Cat_{\infty}$ whose objects are additive $\infty$-categories and morphisms are product preserving 
(i.e. additive) functors is denoted by $Cat^{add}_{\infty}$. 

Now we introduce the notions of an idempotent-complete additive $\infty$-category, the Karoubization of an additive $\infty$-category, 
and the small envelope of an additive category. 
An example to keep in mind is the category of free modules over a ring $R$. It is always additive but it is idempotent complete if and 
only if every projective module over $R$ is free. 
The Karoubization of this category is the category of projective modules whereas its small envelope is the category of stably free modules. 

\begin{defi}\label{krclod}
An additive infinity-category $\au$ is called idempotent-complete (or absolutely Karoubi-closed) if its homotopy category is idempotent 
complete, that is every idempotent $X \stackrel{p}\to X$ in $h(\au)$ has the form
$X \cong X_1 \oplus X_2 \stackrel{\begin{pmatrix} \id_{X_1} & 0 \\ 0 & 0 \end{pmatrix}}\to X_1 \oplus X_2 \cong X$ 
for some $X_1,X_2 \in \obj \au$. 

There exists a universal additive functor $\au \to \kar(\au)$ to an idempotent complete additive $\infty$-category called the 
Karoubization of $\au$ (see Proposition 5.1.4.2 of \cite{ha}). 

The full subcategory of $\kar(\au)$ whose objects are such $X$ that there exist $X', Y \in \obj \au$ with $X\oplus X' \cong Y$ is called the small envelope 
of $\au$. 

A full subcategory $\bu$ of an additive infinity-category $\au$ is called Karoubi-closed (in $\bu$) if any objects $X,Y \in \obj \bu$ such that 
$X\oplus Y \in \obj \au$ also belong to $\obj \bu$.
\end{defi}

%

\begin{pr}[\cite{sag}, Proposition C.1.5.7, Remark C.1.5.9]\label{addspectral}
For an additive $\infty$-category $\au$ 
the functor $\Omega^{\infty}:Fun_{add}(\au^{op}, \spt^{cn}) \to Fun_{add}(\au^{op}, \ssets)$ is an equivalence. 
In particular, there is a fully faithful functor $j:\au\to Fun(\au^{op},\spt)$ such that the usual Yoneda embedding functor is the composition of 
$\Omega^{\infty}$  with $j$.
\end{pr}

For any objects $X,Y$ of an additive $\infty$-category $\au$ we denote the spectrum object $j(X)(Y)$ by $\map(X,Y)$.

For any additive $\infty$-category $\au$ we denote by $Fun^{fin}(\au^{op},\spt)$ the minimal full subcategory of $Fun_{add}(\au^{op},\spt)$ containing the image of the Yoneda 
embedding functor $\au \stackrel{j}\to Fun_{add}(\au^{op},\spt)$ and closed under finite limits and colimits.






\subsection{Weight structures}

Now we recall the definition and some basic properties of weight structures.

\begin{defi}\label{dwstr}
A pair of subclasses $C_{w\le 0},C_{w\ge 0}\subset\obj C$ 
will be said to define a weight
structure $w$ for a triangulated category  $C$ if 
they  satisfy the following conditions:

(i) $C_{w\ge 0},C_{w\le 0}$ are 
Karoubi-closed in $C$. 

(ii) {\bf Semi-invariance with respect to translations.}

$C_{w\le 0}\subset C_{w\le 0}[1]$, $C_{w\ge 0}[1]\subset
C_{w\ge 0}$.

(iii) {\bf Orthogonality.}

$C_{w\le 0}\perp C_{w\ge 0}[1]$.

(iv) {\bf Weight decompositions}.

 For any $M\in\obj C$ there
exists a distinguished triangle
\begin{equation}\label{wd}
X\to M\to Y
{\to} X[1]
\end{equation} 
such that $X\in C_{w\le 0},\  Y\in C_{w\ge 0}[1]$.

\end{defi}

The main example of a weight structure is $C=K(A)$, the homotopy category of complexes over an additive category. In this case 
$K(A)_{w\le 0}$ (resp., $K(A)_{w\ge 0}$) is the class of complexes homotopy equivalent to complexes concentrated in non-positive (resp. 
non-negative) degrees. The weight decomposition axiom is then given by the stupid filtrations of a complex. Unlike the case of t-structures 
already in this simple example weight decompositions are not functorial and not even unique. 
Moreover, bounded from below, from above, or from both sides complexes also give examples of categories with weight structures. These categories are denoted by $K^+(A), K^-(A)$, and  $K^b(A)$, respectively.

Another example that we keep in mind is the spherical weight structure on the stable homotopy category $SH$. 
The classes $SH_{w\le 0}$ and $SH_{w\ge 0}$ are defined as the minimal subcategories containing the sphere spectrum, closed 
under extensions, 
under taking small coproducts, and under taking the negative (resp. positive) triangulated shift. 
This weight structure restricts to the subcategory of compact objects $SH^c$. We refer to section 4.6 of \cite{bws} for details about 
this example. 

\begin{notat}
\begin{itemize}
\item
The full subcategory $\hw\subset \cu$ whose objects are 
$C_{w=0}=C_{w\ge 0}\cap C_{w\le 0}$ 
 will be called the {\it heart} of 
$w$.


\item
 $C_{w\ge i}$ (resp. $C_{w\le i}$, resp.
$C_{w= i}$) will denote $C_{w\ge
0}[i]$ (resp. $C_{w\le 0}[i]$, resp. $C_{w= 0}[i]$).

\item
 The class $C_{w\ge i}\cap C_{w\le j}$ will be denoted 
by $C_{[i,j]}$. 

$C^b\subset C$ is the full subcategory of $C$ whose objects are $\cup_{i,j\in \z}C_{[i,j]}$.

\item
 We  say that $(C,w)$ is {\it  bounded}  if $C^b=C$.

\item
 Let $C$ and $C'$ 
be triangulated categories endowed with
weight structures $w$ and
 $w'$, respectively; let $F:C\to C'$ be an exact functor.

$F$ will be called {\it left weight-exact} 
(with respect to $w,w'$) if it maps
$C_{w\le 0}$ into $C'_{w'\le 0}$; it will be called {\it right weight-exact} if it
maps $C_{w\ge 0}$ into $C'_{w'\ge 0}$. $F$ is called {\it weight-exact}
if it is both left 
and right weight-exact.

\item
 Let $H$ be a 
full subcategory of a triangulated category $C$.

We will say that $H$ is {\it negative} if
$\obj H\perp (\cup_{i>0}\obj (H[i]))$. 
\end{itemize}
\end{notat}

In this paper we will mostly focus on weight structures on the homotopy category of a stable $\infty$-category $\cu$. 
Sometimes we will call a weight structure on $h(\cu)$ just a weight structure on $\cu$. 

\begin{rema}\label{wtrstuff}
\begin{enumerate}
\item
Let $w$ be a bounded weight structure on $h(\cu)$. Then the heart of $w$ generates $h(\cu)$ as a triangulated category or, equivalently,  
$\cu$ is the minimal subcategory of $\cu$ containing objects of $\hw$ and closed under finite limits and colimits (see Corollary 1.5.7 of \cite{bws}). 

\item
The heart of a weight structure is a negative subcategory by definition. Moreover, any negative subcategory $H$ that generates $h(\cu)$ as a triangulated category yields a bounded weight structure whose heart is equivalent to the small envelope of $H$ (see Definition 4.3.1 and Theorem 4.3.2(II.2) of \cite{bws}). 

Consequently, a functor $F$ between $\cu$ and $\cu'$ with bounded weight structures $w$ and $w'$, respectively, 
is weight exact if and only if $F$ maps objects of the heart of $w$ into objects of the heart of $w'$.

\end{enumerate}
\end{rema}

\begin{notat}
\begin{itemize}
\item 
Let $w,w'$ be bounded weight structures on $h(\cu)$ and $h(\cu')$ for a stable $\infty$-category $\cu'$. Then we denote by $Fun_{w.ex}(\cu,\cu')$ the full subcategory of $Fun_{ex}(\cu,\cu')$ whose objects are functors such that their associated functor on the homotopy categories is weight exact. 

\item
We denote by $WCat^{st}_{\infty}$ the simplicial subcategory of $Cat_{\infty}$ whose objects are small stable infinity categories together with a weight structure and the simplicial subset of morphisms between $(\cu,w)$ and $(\cu',w')$ is $Fun_{w.ex}(\cu,\cu')$. 
The full subcategory of $WCat^{st}_{\infty}$ whose objects are stable categories with bounded weight structures is denoted by $WCat^{st,b}_{\infty}$.

\end{itemize}
\end{notat}

\section{Reminder on K-theory spectra}\label{2}
In the section we recall the definitions of the K-theory spectra and state their main properties.

The connective algebraic K-theory spectrum of an exact category (and in particular, of an additive category) was first introduced by Quillen in \cite{quillen}. 
In \cite{waldhausen} Waldhausen defines the formalism of categories with cofibrations and weak equivalences (nowadays called 
Waldhausen categories) and constructs the K-theory spectrum using the so-called $S$-construction. Any exact category gives rise to a 
Waldhausen category and in this case the Waldhausen K-theory spectrum is shown to be homotopy equivalent to Quillen's K-theory 
spectrum (see Appendix 1.9 in ibid.). 

The non-connective K-theory of schemes was studied in \cite{kthder} (see chapter 6). 
Later the non-connective K-theory spectrum of a Frobenius pair (and, in particular, of an exact category) was defined in \cite{schlichting} (see 11.4). 
For a Frobenius pair associated to an exact category its connective cover coincides with the connective Quillen K-theory spectrum. 
Finally, in \cite{highalgkth}  the connective and nonconnective K-theory spectra of a stable $\infty$-category were defined in sections 7 and 9, 
respectively. 

Now we give the definition of the nonconnective K-theory spectrum. 
For this it is important to recall the following statement, usually called the Eilenberg swindle. 

\begin{pr}[\cite{baralgk}, Proposition 8.1]\label{eilswi}
Let $\cu$ be a stable $\infty$-category such that all countable products exist in $\cu$. 
Then $\kthe^{con}(\cu)$ is contractible. 
\end{pr}

As before, let $\cu$ be a small stable $\infty$-category. 
The idea is to embed $\cu$ into a stable category whose K-theory is trivial, take the question, and then iterate the procedure. 
The natural stable category with this property where we can embed $\cu$ is the category of ind-objects $Ind(\cu)$. 
However, this category is usually large, and to get a small category we will take the full subcategory $(Ind(\cu))^{\kappa}$ 
of $Ind(\cu)$ 
whose objects are $\kappa$-compact objects, where $\kappa$ is one's favorite uncountable cardinal. 
Clearly the full embedding $\cu \to Fun(\cu^{op},Sets)$ factors through this subcategory. 
Moreover, the category $(Ind(\cu))^{\kappa}$ is stable (Corollary 1.1.3.6 of \cite{ha}). Denote by 
$\Sigma^1(\cu)$ the localization $(Ind(\cu))^{\kappa}/\cu$. 
Inductively, we define $\Sigma^n(\cu)$ as $\Sigma(\Sigma^{n-1}(\cu))$. 
By \ref{eilswi} $\kthe((Ind(\cu))^{\kappa})$ is homotopy equivalent to the point. 
By functoriality properties of the connective K-theory we have the following commutative diagram of spectra
\begin{center}
\begin{tikzcd}
\kthe^{con}(\Sigma^n(\cu)) \arrow[r]\arrow[d] & \kthe^{con}(Ind(\Sigma^n(\cu))^{\kappa})\cong pt\arrow[d] \\
pt \arrow[r] & \kthe^{con}(\Sigma^{n+1}(\cu))
\end{tikzcd}
\end{center}

Since $\Omega\kthe^{con}(\Sigma^{n+1}(\cu))$ is the homotopy pullback in the diagram above, 
there is a canonical map $\kthe^{con}(\Sigma^n(\cu)) \to \Omega\kthe^{con}(\Sigma^{n+1}(\cu))$. 
We define the non-connective K-theory spectrum $\kthe(\cu)$ as the colimit of these maps 
$\operatorname{colim}\limits_{n\in \mathbb{N}} \kthe^{con}(\Sigma^n(\cu))$.


The main property of the non-connective algebraic K-theory spectrum is the following localization theorem. 

\begin{theo}[\cite{highalgkth},  9.8]\label{localizat}
Let $\cu_1 \to \cu \to \cu_2$ be an exact sequence of idempotent complete stable $\infty$-categories, that is 
$\cu_1 \to \cu_2$ is a full embedding, the composition is trivial and the induced map $Kar(\cu/\cu_1) \to \cu_2$ is an equivalence. 

Then the sequence
$\kthe(\cu_1) \to \kthe(\cu) \to \kthe(\cu_2)$ is a cofiber sequence.
\end{theo}

\begin{notat}
\begin{itemize}
\item
For an $\infty$-category $\cu$ 
we denote by $\kthe_i(\cu)$ the $i$-th homotopy groups of the spectrum $\kthe(\cu)$. 

Note that the K-theory spectrum in \cite{highalgkth} is defined to be Morita invariant. 
So be warned that $\kthe_0(h(\cu))$ (i.e. $\kthe_0$ of the triangulated category $h(\cu)$) does not coincide 
with $\kthe_0(\cu) = \pi_0(\kthe(\cu))$ unless $\cu$ is idempotent complete.

By $\kthe(-)$ we always mean the nonconnective K-theory spectrum. We denote the connective K-theory spectrum by $\kthe^{con}$.

\item
For an additive $\infty$-category $\au$ we denote by $\kthe(\au)$ the spectrum $\kthe(Fun^{fin}(\au^{op},\spt))$. 
If $H$ is a classical additive category $\kthe(H)$ is defined to be $\kthe(\underline{Nerve}(Com^b(H)))$. Later we will see that $\kthe(H)$ is equivalent to $\kthe(Nerve(H))$. 

\end{itemize}
\end{notat}

\section{Weight complex functor}\label{3}
From now on $\cu$ will be a stable $\infty$-category together with a weight structure $w$ on $h(\cu)$.

\begin{defi}
The $\infty$-heart $\hwoo$ of $w$ is the full subcategory of $\cu$ whose objects are those of $\hw$. 
 It is an additive 
$\infty$-category.
\end{defi}

\begin{lem}\label{adjstadd}
For any additive $\infty$-category $\au$ and any stable $\infty$-category $\cu$ 
the restriction functor 
$Fun^L(Fun_{add}(\au^{op}, \spt),\cu) \to Fun_{add}(\au,\cu)$ is an equivalence. 
\end{lem}
\begin{proof}
By Remark C.1.5.9 of \cite{sag} 
we can identify $Fun_{add}(\au^{op}, \spt)$ with the stabilization of $Fun_{add}(\au^{op}, \ssets)$. 
Hence by Corollary 1.4.4.5 of \cite{ha} the restriction functor 
$$Fun^L(Fun_{add}(\au^{op}, \spt), \cu) \to Fun^L(Fun_{add}(\au^{op}, \ssets), \cu)$$ 
is an equivalence.

Now let $\mathcal{K}$ be the collection of all small simplicial sets 
and $\mathcal{R}$ be the collection of all maps from finite discrete simplicial sets to $\au$.  
Then by Proposition 5.3.6.2(2) of \cite{htt} applied to this setting of 
the restriction functor 
$$Fun^L(P(\au), \cu) \to Fun_{add}(\au, \cu)$$ 
is also an equivalence. 
\end{proof}


We now prove that a stable $\infty$-category with a bounded weight structure is determined by its heart.

\begin{pr}\label{adeqw}
Let $w$ be bounded.

1. The essential image of the composition functor  
$F':\cu \stackrel{j}\to Fun_{ex}(\cu^{op}, \spt) \stackrel{res}\to Fun_{add}(\hwoo^{op}, \spt)$ lies in 
the full subcategory $Fun^{fin}(\hwoo^{op}, \spt)$ and the functor $\cu \stackrel{F}\to Fun^{fin}(\hwoo^{op}, \spt)$ is an equivalence of 
$\infty$-categories.

2.
Let $\cu'$ be a stable $\infty$-category with a bounded weight structure $w'$. Then the restriction functor 
$Fun_{w.ex}(\cu,\cu') \to Fun_{add}(\hwoo,\underline{Hw'}_{\infty})$ 
is an equivalence of $\infty$-categories.
\end{pr}
\begin{proof}
1.
Let $X,Y$ be objects of $\hwoo$. 
By definition $F'$ maps $X$ to the functor $\map_{\cu}(-,X)$ restricted to the subcategory $\hwoo^{op}$ of $\cu^{op}$. 
By the axiom (iii) of weight structures $\map_{\cu}(-,X)$ is a connective spectrum. Clearly, $\Omega^{\infty}\map_{\cu}(-,X) = Map_{\cu}(-,X) = 
j(X)$. 
By Proposition \ref{addspectral} the map 
$Map_{Fun_{add}(\hwoo^{op}, \spt)}(F'(X),F'(Y)) \stackrel{\Omega^{\infty}}\to Map_{Fun_{add}(\hwoo^{op}, \ssets)}(j(X),j(Y))$ 
 is an 
equivalence. By the Yoneda Lemma the map 
$Map_{\cu}(X,Y) \to Map_{Fun_{add}(\hwoo^{op}, \spt)}(j(X),j(Y))$ 
is an equivalence. 
Since $\Omega^{\infty} \circ F'|_{\hwoo} = j$  the map  $Map_{\cu}(X,Y) \to Map_{Fun_{add}(\hwoo^{op}, \spt)}(F'(X),F'(Y))$ is also an equivalence. Now it is clear that $F'|_{\hwoo}$ is a full embedding whose essential image is equal to $j(\hwoo)$. 

By Proposition 1.1.4.1 of \cite {ha} $F'$ commutes with finite limits and finite colimits. By Remark \ref{wtrstuff}(1) any object of $\cu$ can be obtained from objects of $\hw_{\infty}$ by taking finite limits and colimits. Hence the essential image of $F'$ lies in $Fun^{fin}(\hwoo^{op}, \spt)$.

Since any object of $Fun^{fin}(\hwoo^{op}, \spt)$ can be obtained from objects of $\hw_{\infty}$ by taking finite limits and colimits it suffices to prove that 
$F$ is a full embedding. 
The morphism $\map_{\cu}(X,Y) \stackrel{F_{X,Y}}\to \map_{Fun^{fin}(\hwoo^{op}, \spt)}(F(X),F(Y))$ 
is an equivalence for any $(X,Y) \in \obj \hwoo^{op} \times \hwoo$. Since any object of $\cu^{op} \times \cu$ can be obtained 
from objects of $\hwoo^{op} \times \hwoo$ by taking finite limits and finite colimits and $F$ is exact, the morphism $F_{X,Y}$ is an equivalence for any $(X,Y) 
\in \obj \cu^{op} \times \cu$, which means that $F$ is a full embedding, and thus, is an equivalence. 


2.
To simplify notations for any additive category $\au$ we denote  the category $Fun_{add}(\au^{op}, \ssets)$ by $P(\au)$ and  $Fun^{(fin)}_{add}(\au^{op}, \spt)$  by $SP^{(fin)}(\au)$. 

By assertion 1 and since equivalences induce equivalences of the corresponding functor categories $Fun(-,-)$ (Proposition 1.2.7.3 of \cite{htt}) we can assume $\cu = SP^{fin}(\hw)$, $\cu' = SP^{fin}(\underline{Hw'}_{\infty})$. 

By Proposition 5.5.8.10(6) of \cite{htt} together with Proposition 1.4.3.7 of \cite{ha} the $\infty$-category $SP(\hwoo)$ is 
compactly generated. 
Proposition 5.3.4.17 of \cite{htt} implies that the subcategory of compact objects is equivalent to the Karoubization of $SP^{fin}(\hwoo)$. 
By definition of a compactly generated $\infty$-category $Ind(SP(\hwoo)^c) = SP(\hwoo)$ (see 5.5.7.1 of \cite{htt}). 
Now the restriction functor 
$$Fun^L(Ind(SP^{fin}(\hwoo)), SP(\underline{Hw'}_{\infty})) \to Fun_{ex}(SP^{fin}(\hwoo), SP(\underline{Hw'}_{\infty}))$$ 
is an equivalence of categories by Proposition 5.1.4.9, Proposition 5.3.5.10, and Proposition 5.3.5.14 of \cite{htt}.

The following diagram commutes
\begin{center}
\begin{tikzcd}
Fun^L(SP(\hwoo), SP(\underline{Hw'}_{\infty})) \arrow[r,"\cong"]\arrow[rd,"\cong"] & 
Fun_{ex}(SP^{fin}(\hwoo), SP(\underline{Hw'}_{\infty})) \arrow[d]\\
& Fun_{add}(\hwoo, SP(\underline{Hw'}_{\infty}))
\end{tikzcd}
\end{center}
where arrows are restriction functors. Above we've proved that the top horizontal functor is an equivalence. 
Moreover, by Lemma \ref{adjstadd} the diagonal functor is an equivalence. 
Hence the functor  
$$Fun_{ex}(SP^{fin}(\hwoo), SP(\underline{Hw'}_{\infty}))\stackrel{Res}\to Fun_{add}(\hwoo, SP(\underline{Hw'}_{\infty}))$$ 
is also an equivalence.


We denote the full embedding  $\underline{Hw'}_{\infty} \to SP(\underline{Hw'}_{\infty})$ by $i$ and the full embedding 
$SP^{fin}(\underline{Hw'}_{\infty}) \to SP(\underline{Hw'}_{\infty})$ by $i'$. 
Let $\mathcal{K} \subset Fun_{ex}(SP^{fin}(\hwoo), SP(\underline{Hw'}_{\infty}))$ be the full subcategory whose objects are 
functors $G$ such that $Res(G) \cong i \circ U$ for some $U:\hwoo\to \underline{Hw'}_{\infty}$. 
Since $Res$ is an equivalence the map $\mathcal{K} \to Fun_{add}(\hwoo, \underline{Hw'}_{\infty})$ is also an equivalence. 


Certainly, the full subcategory $Fun_{w.ex}(SP^{fin}(\hwoo), SP^{fin}(\underline{Hw'}_{\infty}))$ is a subset of $\mathcal{K}$. But also the converse is true. Indeed, let $G$ be a functor 
$SP^{fin}(\hwoo) \to SP(\underline{Hw'}_{\infty})$ such that the image of any object of $\hwoo$ is an object of $\underline{Hw'}_{\infty}$. 
By definition any object $X$ of $SP^{fin}(\hwoo)$ can obtained via a finite combination of limits and colimits from objects of $\hwoo$. 
Since the functor $G$ is exact,  $G(X)$ is an object of the subcategory $SP^{fin}(\underline{Hw'}_{\infty})$. 
So we obtain a functor $G':SP^{fin}(\hwoo) \to SP^{fin}(\underline{Hw'}_{\infty})$ and it's clearly weight exact. 

Now the restriction map from 
$$Fun_{w.ex}(SP^{fin}(\hwoo), SP^{fin}(\underline{Hw'}_{\infty})) \cong \mathcal{K}' = \mathcal{K}$$ 
to $Fun_{add}(\hwoo, \underline{Hw'}_{\infty})$ is an equivalence.

\end{proof}



\begin{coro}\label{wstrares}
The functor $WCat^{st,b}_{\infty} \to Cat^{add}_{\infty}$ that sends a stable $\infty$-category with a bounded weight structure to 
its heart, is a full embedding of categories enriched over quasi-categories. The essential image consists of those additive categories whose homotopy categories coincide with their small envelope. 
\end{coro}
\begin{proof}
The functor is a full embedding by Proposition \ref{adeqw}(2). 

All the additive categories in the image coincide with their small envelope by Theorem 4.3.2(II.2) of \cite{bws}. 
Conversely, let $\hu$ be an additive infinity-category. The functor $\hu \stackrel{j}\to Fun^{fin}(\hu^{op},\spt)$ is a full embedding by Proposition \ref{addspectral}. 
The subcategory $h(j(\hu))$ is negative and it generates $h(Fun^{fin}(\hu^{op},\spt))$. So, by Remark \ref{wtrstuff}(2) there exists a bounded weight structure on $Fun^{fin}(\hu^{op},\spt)$ whose $\infty$-heart is equivalent to the small envelope of $j(\hu)$. 
\end{proof}

\begin{coro}\label{wecom}
Let $w$ be bounded. Then there exists an exact functor $\cu \stackrel{t}\to Com^{b}(\hw)$ unique up to homotopy that induces 
a weight exact functor $h(\cu) \stackrel{h(t)}\to K^b(\hw)$ that is the identity restricted to the hearts of the weight structures. 

The functor $h(\cu) \stackrel{t}\to K^b_w(\hw) \to K^b_w(\hw)$ is the weight complex functor in the sense of \cite{bws}(3) (where $K^b_w(\hw)$ is the weak homotopy category of complexes). 

For any stable $\infty$-category with a bounded weight structure $w'$ and an exact functor $\cu \stackrel{G}\to \cu'$ the following diagram commutes 
$$\begin{CD} 
\cu @>{t}>> Com^{b}(\hw) \\ 
 @VV{G}V  @VV{}V \\
\cu' @>{t}>> Com^{b}(\underline{Hw'}) \\
\end{CD}$$
\end{coro}
\begin{proof}
Consider the functor $u:Cat^{add}_{\infty} \stackrel{Nerve(h(-))}\to Cat^{add}_{\infty}$ which is the unit 
$\id_{Cat^{add}_{\infty}} \to Nerve(h(-))$ 
of the adjunction 
\begin{tikzcd}
Cat^{add}_{\infty} \arrow[r,shift left=.5ex,"h"]& Cat^{add}\arrow[l,shift left=.5ex,"Nerve"]
\end{tikzcd}.
By the equivalence from Corollary \ref{wstrares} we know that the restriction map 
$Fun_{w.ex}(\cu,Com^{b}(\hw)) \to Fun_{add}(\hwoo,Nerve(\hw))$ 
is an equivalence of infinity-categories for any $\cu$ with a bounded weight structure $w$.  
In particular there exists a functor $\cu \stackrel{t}\to Com^{b}(\hw)$ whose restriction to the hearts is $u$. 
The functor  $h(\cu) \stackrel{t}\to K^b(\hw) \to K^b_w(\hw)$ is isomorphic to the functor defined in \cite{bws}(3) because $t$ is compatible with weight Postnikov towers. 

Moreover, the functor $\au \to Nerve(h(\au))$ is the unique functor up to homotopy that induces the identity functor $h(\au) \to h(\au)$. 
Thus the uniqueness of $t$ follows.

\end{proof}

\begin{rema}
The corollary above solves Conjecture 3.3.3 of \cite{bws} for triangulated categories with a bounded weight structure having an $\infty$-categorical enhancement. 

Moreover, it enables us to solve the conjecture for enhanced triangulated categories with a compactly-generated weight structure. 
Let $C$ be a $\kappa$-compactly generated triangulated category with a compactly generated weight structure $w$ on it (i.e. the heart contains the set of compact generators of the category). 
We assume that it has an $\infty$-categorical model $\cu$. By Remark 1.4.4.3 $\cu$ is compactly generated and 
in particular, the functor $Ind(\cu^c) \stackrel{e}\to \cu$ is an equivalence. 
By Proposition 5.3.5.14 of \cite{htt} the functor $e$ commutes with colimits. 
By Proposition 1.1.3.6 of \cite{ha} the category $Ind(\cu^c)$ is stable and hence the functor $e$ is exact. The triangulated functor 
$h(Ind(\cu^c)) \stackrel{h(e)}\to C$ is now an equivalence since it induces an equivalence on the subcategories of compact objects. 

Using Proposition 5.3.5.10 of \cite{htt} again and Corollary \ref{wecom} we obtain a functor $C \cong h(Ind(\cu^c)) \to h(Ind (Com^b(\hw^c))) \cong h(Com(\hw)) \hookrightarrow K(\hw)$. 
\end{rema}

\section{Theorems of the heart}\label{4}

Now we want to relate the K-theory of $\cu$ and the K-theory of its heart $\hw$. 
Unlike the situation with t-structures, the $\infty$-heart $\hw_{\infty}$ might have non-trivial higher homotopy groups.
So, an $\infty$-category theorist would say that the correct statement of the theorem of the heart should be the following

\begin{coro}[The stupid theorem of the heart for weight structures]\label{stthoftheheart}
There is a canonical homotopy equivalence $\kthe(\hwoo) \to \kthe(\cu)$.
\end{coro}

Since we already know that $\cu \cong Fun^{fin}(\hwoo^{op}, \spt)$ (see Proposition \ref{adeqw}(1)) the statement is obvious
\footnote{The author was informed that Ernie Fontes had been working on a related statement. Unfortunately, the author couldn't get any 
precise information about the work or the statement}.


However, we still want to compare $\kthe(\cu)$ and $\kthe(\hw)$ because $\kthe(\hw)$ is a priori something easier.

Note that in general the spectrum $\kthe(\hwoo)$ is quite distinct from $\kthe(\hw)$. 
For example, the category of compact objects of the stable homotopy category $SH^{c}$ possesses a weight structure whose $\infty$-heart $
\hwoo$ is the additive subcategory generated by the sphere spectrum. So $\kthe(\hwoo) \cong \kthe(\mathbb{S}-mod^{fin}) = \kthe(\mathbb{S})$. 
The classical heart of this weight structure is equivalent to the category of finitely generated free $\z$-modules, so $\kthe(\hw) \cong 
\kthe(\mathbb{Z})$. 
The groups $\kthe_i(\mathbb{S})$ and $\kthe_i(\mathbb{Z})$ are well-known to not be isomorphic. For instance, it was shown 
that the $p$-torsion of $\kthe_i(\mathbb{S})$ contains the $p$-torsion of $\pi_i(\mathbb{S})$ for an odd prime $p$ (Theorem 1.2 in \cite{homksph}, see also \cite{algkss}). 

However, in some cases $\kthe_i(\hw)$ is actually isomorphic to $\kthe_i(\hwoo)$. 
Consider the map of K-theory spectra $\kthe(\cu) \to \kthe(\hw)$ induced by the weight complex functor (constructed in Corollary 
\ref{wecom}). 
First note that for $i=0$ the induced map $\pi_0(\kthe(\cu)) \to \kthe_0(Kar(h(\cu)))$ is an isomorphism. 
This together with Theorem 5.3.1 of \cite{bws} implies the following. 

\begin{pr}\label{amap}
The map of K-theory spectra $\kthe(\cu) \to \kthe(\hw)$ induces isomorphism in $\pi_0$.
\end{pr}

Now we generalize this result to all the negative K-groups.

\begin{theo}[The theorem of the heart for weight structures in negative K-theory]\label{mainwih}
The map $\kthe_i(\cu) \to \kthe_i(\hw)$ is an isomorphism for $i \le 0$.
\end{theo}
\begin{proof}
The proof goes by decreasing induction over $i$. For $i =0$ the statement follows from Proposition \ref{amap}. 

Assume now the theorem is known for $n\ge i+1$. 
Denote by $\hwoo^{big}$ the full subcategory of $Fun(\cu^{op}, \spt)$ that contains $\bigoplus\limits_{i\in \mathbb{N}} X_i$ for any sequence 
of objects $X_i$ of $\hwoo$. Next denote by $\cu^{big}$ the smallest full subcategory of $Fun(\cu^{op}, \spt)$ containing $\hwoo^{big}$ and closed 
under taking finite limits and colimits. 
For any object $X$ of $\hwoo^{big}$ the coproduct $\bigoplus_{i \in \mathbb{N}} X$ exists in $\cu^{big}$. 
Indeed, let $X$ be an object of $\cu^{big}$. For some $n$ its shift $\Sigma^n X$ is a colimit of objects of $\hwoo^{big}$. Coproducts 
commute with colimits and with $\Sigma^n$ (since $\Sigma$ is an equivalence), hence the coproduct $\bigoplus_{i \in \mathbb{N}} X$ also exists in $
\cu^{big}$. 
The proof of Proposition 8.1 of \cite{baralgk} only uses the existence of such coproducts in the category. So, it yields that $
\kthe(\cu^{big})$ is homotopy equivalent to the point.
Denote by $\hw^{big}$ the homotopy category of $\hwoo^{big}$. Using the same argument we obtain that $\kthe(Com^b(\hw^{big}))$ is homotopy equivalent to the point. 
By definition the category $\cu^{big}$ admits a weight structure whose heart is $\hwoo^{big}$. 

Now we use Corollary \ref{wecom} to form the diagram 
$$\begin{CD} 
\cu @>{}>> \cu^{big} @>{}>> \kar(\cu^{big}/\cu)\\ 
 @VV{t_{\cu}}V @VV{t_{\cu^{big}}}V @VV{t'}V \\
Com^b(\hw) @>{}>> Com^b(\hw^{big}) @>{}>> \cu' = \kar(Com^b(\hw^{big})/Com^b(\hw))\\
\end{CD}$$
By Theorem \ref{localizat} it induces the following diagram of K-groups whose rows are exact sequences 
$$\begin{CD} 
\kthe_{i+1}(\cu^{big}) @>{}>> \kthe_{i+1}(\kar(\cu^{big}/\cu)) @>{d_1}>>  \kthe_i(\cu) @>{}>> \kthe_i(\cu^{big})\\ 
           @VV{}V                                      @VV{\kthe_{i+1}(t')}V                    @VV{\kthe_i(t_{\cu})}V                  @VV{}V \\
\kthe_{i+1}(\hw^{big}) @>{}>> \kthe_{i+1}(\cu') @>{d_2}>> \kthe_i(\hw) @>{}>> \kthe_i(\hw^{big}) \\
\end{CD}$$
By Theorem 8.1.1 of \cite{bws} $h(\cu^{big}/\cu)$ and $h(\cu')$ admit weight structures with the heart $\frac{\hw^{big}}{\hw}$. The functor 
$t'$ is weight-exact and it induces an 
identity functor on the hearts. 
One may notice that $t_{\cu'} \circ t' \cong t_{\cu^{big}/\cu}$ (see the uniqueness statement in Corollary \ref{wecom}) 
where $t_{\cu^{big}/\cu}$ and $t_{\cu'}$ are corresponding weight complex functors. 
 Then by the inductive assumption $\kthe_{i+1}(t_{\cu'})$ and $\kthe_{i+1}(t_{\cu^{big}/\cu})$ are isomorphisms, hence so is $\kthe_{i+1}(t')$.  
Since $\kthe_n(\cu^{big}) \cong \kthe_n(\hw^{big}) \cong 0$ for any $n$  the maps $d_1$ and $d_2$ are also isomorphisms. 
Hence $\kthe_i(t_{\cu})$ is also an isomorphism.

\end{proof}

\begin{rema}
1. Theorem \ref{mainwih} can be seen as a generalization of Theorem 9.53 of \cite{highalgkth} from additive $\infty$-categories generated by 
one object to general additive $\infty$-categories. Moreover, modulo Corollary \ref{stthoftheheart} or the main result of \cite{schwedeshipley} 
one could derive our result from their theorem and from the fact that K-theory commutes with filtered colimits. The ideas of the two proofs are 
essentially the same.

2. Presumably the category $Com^b(\hw^{big})/Com^b(\hw)$ appearing in the proof of \ref{mainwih} is equivalent to 
$Com^b(\hw^{big}/\hw)$. That is, not only $\kthe_{i+1}(t')$ is an isomorphism but also $t'$ itself is an equivalence. 
However, since the proof of this fact is unnecessary and requires some work on localizations of triangulated categories, we don't include it 
into the exposition.

\end{rema}





\section{Negative K-theory of motivic categories}\label{5}
The groups $\kthe_{i}(\hw)$ are much easier to compute than $\kthe_{i}(\cu)$. 
Indeed, let $A$ be an idempotent complete additive category. For any object $M \in \obj A$ the full additive subcategory generated by $M$ is equivalent to the category $free(\enom_{A}(M))$ of finitely-generated right free modules over $\enom_{A}(M)$. Moreover,  the full 
additive subcategory closed under retracts generated by $M$ is equivalent to the category $proj(\enom_{A}(M))$ of finitely-generated right projective modules over $\enom_{A}(M)$. 
This certainly implies that the additive category $A$ is equivalent to the filtered colimit of categories $proj(\enom_{A}(M))$. 
By Corollary 6.4 of \cite{schlichting} $\kthe_i(A) \cong \operatorname{colim}_{M \in \obj A} \kthe_i(\enom_A(M))$ for any idempotent complete additive category $A$, where the colimit is taken with respect to the maps induced by the embeddings of minimal Karoubi-closed additive subcategories of $A$ containing $M$.

Now let $\dmges(k; R)$ denote the category of compact objects in the category of effective Voevodsky motives over a field $k$ 
with coefficients in a ring $R$. 
Assume also that $\operatorname{char}(k)$ is invertible in $R$. 
By the results of section 6.5 of \cite{bws} there exists a bounded weight structure on this category whose heart is the 
category of Chow motives. 

From Voevodsky's construction it is clear that $\dmges(k;R)$ admits an $\infty$-enhancement 
(see also the paper \cite{beilvol} for the construction of a thorough of a dg-enhancement). 
We denote the corresponding 
stable $\infty$-category 
by $\underline{\dmges}(k;R)$. 
Now Theorem \ref{mainwih} yields the following generalization of Theorem 6.4.2 of \cite{wcom}.

\begin{pr}\label{kchwo}
The weight complex functor for  $\underline{\dmges}(k;R)$ 
induces isomorphisms 
$\kthe_n(\underline{\dmges}(k;R)) \to \kthe_n(\chow^{eff}(k;R))$ for any $n\le 0$.
\end{pr}

A hard conjecture predicts that for $R = \q$ there exists a bounded t-structure on this category with certain properties called
 the motivic t-structure.  For general rings of coefficients the conjecture doesn't hold (see Proposition 4.3.8 of \cite{voevtrimot}). 

\begin{coro}\label{mottstr}
If the motivic t-structure exists on $\dmges(k;R)$ then $\kthe_{-1}(\chow^{eff}(k;R)) = 0$. 
If, moreover, the generalized Schlichting's conjecture is true (see \cite{thofneghrt}), then $\kthe_{n}(\chow^{eff}(k;R)) = 0$ for all $n<0$.
\end{coro}
\begin{proof}
By the main result of \cite{thofneghrt} the groups $\kthe_n(\cu)$ are zero for $n=-1$ and for $n<-1$ if the generalized Schlichting's conjecture is true. 
Hence by Proposition \ref{kchwo} $\kthe_n(\chow^{eff}(S;R))$ are zero.
\end{proof}

\begin{rema}
Note that the homotopy t-structure on $DM^{eff}(k)$ does not restrict to the subcategory $DM^{eff}_{gm}(k)$, so our theorem 
cannot be applied. 
\end{rema}

\begin{rema}
For $R =\q$ the vanishing of the negative K-groups of  $\chow(k;\q)$ also follows from the smash-nilpotence conjecture of Voevodsky. 
Indeed, assume the smash-nilpotence conjecture holds. Then by Corollary 3.3 of \cite{voevod} the ring of endomorphisms of any motive $M \in \obj \chow(k;\q)$ is a nil-extension of the ring of endomorphisms of the image of $M$ in the category of Grothendieck motives $GM(k;\q)$. 
Since $GM(k;\q)$ is semisimple abelian, $\enom_{GM(k;\q)}(M)$ and $\enom_{\chow(k;\q)}(M)$ are artinian rings. 
Thus, their negative K-theory vanishes (see Proposition 10.1 of \cite{bass}.XII). 
Now using the representation of $\kthe_i(\chow(k;\q))$ as the colimit of $\kthe_i(\enom_{\chow(k;\q)}(M))$ 
we obtain that $\kthe_i(\chow(k;\q)) = 0$.
\end{rema}

\begin{question}
Voevodsky has shown that the motivic t-structure does not exist on $\dmges(k;\z)$ if there is a conic without rational points over $k$ (see Proposition 4.3.8 of \cite{voevtrimot}). Is it also possible to see this finding an obstruction to having one in $\kthe_{-1}(\chow(k;\z))$?
\end{question}

\end{document}